\documentclass{scrartcl}

\usepackage[english]{babel}
\usepackage[utf8]{inputenc}
\usepackage[maxbibnames=99, giveninits]{biblatex}       
\usepackage{csquotes}       
\usepackage{amsthm}         
\usepackage{mathtools}      
\usepackage{amsfonts}       
\usepackage{enumitem}
\usepackage{tikz}           
\usepackage{subdepth}       
\usepackage{xcolor}         
\usepackage{xfrac}

\usetikzlibrary{positioning, arrows.meta, calc}
\tikzstyle{graph} = [node distance = 1.5cm, every path/.style = {semithick}]
\tikzset{graph vertex/.style = {inner sep = 1pt, circle, fill}}
\tikzstyle{graph vertices} = [every node/.style = graph vertex]

\newtheorem{thm}{Theorem}[section]
\newtheorem{cor}[thm]{Corollary}
\newtheorem{lem}[thm]{Lemma}
\newtheorem{prop}[thm]{Proposition}

\theoremstyle{definition}

\newtheorem{exam}[thm]{Example}

\numberwithin{equation}{section}    

\DeclarePairedDelimiter\br{(}{)}

\DeclarePairedDelimiter\floor{\lfloor}{\rfloor}
\DeclarePairedDelimiter\abs{|}{|}

\DeclareMathOperator{\Sz}{Sz}
\DeclareMathOperator{\dis}{dis}
\DeclareMathOperator{\deq}{deq}

\title{Comparing Wiener, Szeged and revised Szeged index on cactus graphs}
\author{Stefan Hammer}
\date{\today}

\begin{document}

\maketitle

\begin{abstract}
    We show that on cactus graphs the Szeged index is bounded above by twice the Wiener index. For the revised Szeged index the situation is reversed if the graph class is further restricted. Namely, if all blocks of a cactus graph are cycles, then its revised Szeged index is bounded below by twice its Wiener index. Additionally, we show that these bounds are sharp and examine the cases of equality. Along the way, we provide a formulation of the revised Szeged index as a sum over vertices, which proves very helpful, and may be interesting in other contexts. 
\end{abstract}

\section{Introduction}

Presumably the first topological graph index, the Wiener index, was invented in 1947 by the chemist \citeauthor{1947_StructuralDeterminationOfParaffinBoilingPoints}~\cite{1947_StructuralDeterminationOfParaffinBoilingPoints}, and is used to correlate physicochemical properties to the structure of chemical compounds \cite{2002_TheWienerNumber, 1986_MathematicalConceptsInOrganicChemistry}. Since then it was and still is thoroughly studied, see e.g. \cite{2021_WienerIndexInGraphsWithGivenMinimumDegreeAndMaximumDegree, 2022_WienerIndexOfBasilicaGraphs, 2021_ProofOfAConjectureOnTheWienerIndexOfEulerianGraphs, 2021_OnTheRelationBetweenWienerIndexAndEccentricityOfAGraph, 2022_OnTheWienerIndexOfTwoFamiliesGeneratedByJoiningAGraphToATree, 2021_AnAsymptoticRelationBetweenTheWirelengthOfAnEmbeddingAndTheWienerIndex} for only some of the latest results. Over time many more topological graph indices were devised and investigated. One such topological graph index is the Szeged index that came up as an extension of a formula for the Wiener index of trees. It was first introduced in \cite{1994_AFormulaForTheWienerNumberOfTreesAndItsExtensionToCyclicGraphsContainingCycles} without proper name. By its construction it has meaningful connections to the Wiener index. However,  \citeauthor*{2002_OnGeneralizationOfWienerIndexForCyclicStructures} found that the Szeged index is lacking something for chemical applications in comparison to the Wiener index, and thus introduced in \cite{2002_OnGeneralizationOfWienerIndexForCyclicStructures} a slightly adapted variant of the Szeged index, the later so-called revised Szeged index. It produces better correlations in chemistry than the normal Szeged index \cite{2002_OnGeneralizationOfWienerIndexForCyclicStructures} and both Szeged indices combined can be used to provide a measure of bipartivity of graphs \cite{2010_UseOfTheSzegedIndexAndTheRevisedSzegedIndexForMeasuringNetworkBipartivity}. 

It is rather easy to see that the Wiener index and the (revised) Szeged index coincide on trees. Furthermore, in 1994 some conjectures about the relation of the Wiener and the Szeged index on connected graphs were made by \citeauthor{1994_OnAGraphInvariantRelatedToTheSumOfAllDistancesInAGraph}  \cite{1994_OnAGraphInvariantRelatedToTheSumOfAllDistancesInAGraph, 1994_AFormulaForTheWienerNumberOfTreesAndItsExtensionToCyclicGraphsContainingCycles}. A year later already \citeauthor{1995_SolvingAProblemConnectedWithDistancesInGraphs} proved that the Wiener index and the Szeged index are equal if and only if every block of the graph is complete \cite{1995_SolvingAProblemConnectedWithDistancesInGraphs}. Another year later \citeauthor*{1996_TheSzegedAndTheWienerIndexOfGraphs} showed that the Szeged index is at least as big as the Wiener index \cite{1996_TheSzegedAndTheWienerIndexOfGraphs}. Since then many more authors investigated the relation of the Wiener and the Szeged index, see \cite{2017_OnTheDifferenceBetweenTheSzegedAndTheWienerIndex, 2014_ImprovedBoundsOnTheDifferenceBetweenTheSzegedIndexAndTheWienerIndex, 2013_WienerIndexVersusSzegedIndexInNetworks, 2012_GraphsWhoseSzegedAndWienerNumbersDifferBy4And5} and references therein. This research has been extended to the revised Szeged index \cite{2014_OnTheDifferenceBetweenTheRevisedSzegedIndexAndTheWienerIndex, 2017_OnTheQuotientsBetweenTheRevisedSzegedIndexAndWienerIndex, 2016_OnTheFurtherRelationBetweenTheSzegedIndexAndTheWienerIndex} and to certain graph classes \cite{2018_OnTheDifferenceBetweenTheRevisedSzegedIndexAndTheWienerIndexOfCacti, 2016_ProofsOfThreeConjecturesOnTheQuotientsOfTheSzegedIndexAndTheWienerIndex, 2022_OnTheDifferenceBetweenTheSzegedIndexAndTheWienerIndexOfCacti}, with the most recent work on cactus graphs dating from this year. For further current research in the context of comparing graph indices with the Wiener index, we refer the interested reader to \cite{2022_ComparisonBetweenMerrifieldSimmonsIndexAndWienerIndexOfGraphs, 2022_ComparativeResultsBetweenTheNumberOfSubtreesAndWienerIndexOfGraphs, 2021_RelationsBetweenMerrifieldSimonsAndWienerIndices}.

In this paper, we want to show new relations between Wiener, Szeged and revised Szeged index for the special case of cactus graphs. Namely, we prove that the Szeged index is bounded above by twice the Wiener index. In case of the revised Szeged index the situation is more complex. For bipartite cacti the revised Szeged is equal to the Szeged index, but if we limit the class of cactus graphs to those that have only cycles as blocks, we can reverse the above statement. That is, the revised Szeged index is bounded below by twice the Wiener index. Additionally, we show that these bounds are sharp and examine the cases of equality. Along the way, we provide a formulation of the revised Szeged index as a sum over vertices, which proves very helpful, and may be interesting in other contexts. 

The paper is organized as follows. In Section~\ref{sec:Preliminaries}, we first introduce the main definitions and directly afterwards show how the revised Szeged index can be written as a sum over vertices (Theorem~\ref{thm:RevisedSzAsVertexSum}). Then we introduce some auxiliary results needed in the following sections. The relation of the Szeged index and the Wiener index on cactus graphs is the main topic of Section~\ref{sec:WIAndSzOnCactusGraphs}. We show that the Szeged index is bounded above by twice the Wiener index (Theorem~\ref{thm:Sz2W}), and also look at equality cases. Section~\ref{sec:WIAndRevisedSzOnCactusGraphs} starts with an example showing that arbitrary cactus graphs can have a revised Szeged index equal to twice its Wiener index. As a consequence, we look at a subclass of the cactus graphs to prove a reverse relation for the revised Szeged and the Wiener index (Theorem~\ref{thm:RevisedSz2W}).

\section{Preliminaries and the revised Szeged index as vertex sum}
\label{sec:Preliminaries}

If not otherwise mentioned, we are working with a finite, simple and connected graph $G$, that has vertex set $V(G)$ and edge set $E(G)$. Let $u$, $v$ be vertices of $G$. Then we denote with $d_G(u, v)$ the distance of $u$ and $v$ in $G$, that is, the length of the shortest path connecting $u$ and $v$ in $G$. For a path $P$, we use $\abs{P}$ for its length. Furthermore, we write $n_G(u, v)$ for the number of vertices closer to $u$ than to $v$, and $o_G(u, v) = o_G(v, u)$ for the number of vertices with equal distance to $u$ and $v$. With this, the \emph{Wiener index}, the \emph{Szeged index}, and the \emph{revised Szeged index} are defined respectively by
\begin{equation*}
    \begin{split}
        W(G) &= \sum_{\{u, v\} \subseteq V(G)} d_G(u, v) = \frac{1}{2} \, \sum_{u, v \in V(G)} d_G(u, v), \\
        \Sz(G) &= \sum_{\{s, t\} \in E(G)} n_G(s, t) \, n_G(t, s) , \\
        \Sz^*(G) &= \sum_{\{s, t\} \in E(G)} 
            \br*{n_G(s, t) + \frac{1}{2} \, o_G(s, t)} 
            \br*{n_G(t, s) + \frac{1}{2} \, o_G(t, s)} . 
    \end{split}
\end{equation*}
Note, that the Wiener index is a sum over all unordered pairs of vertices, whereas the (revised) Szeged index is a sum over all edges. 

In \cite{2000_SomeGraphsWithExtremalSzegedIndex}, \citeauthor*{2000_SomeGraphsWithExtremalSzegedIndex} introduced for vertices $u$, $v$ and an edge $\{s, t\}$ the function 
\begin{equation*}
    \mu_{u, v}(\{s, t\}) =
        \begin{cases}
            1 &\text{if} 
                \begin{cases}
                    d_G(u, s) < d_G(u, t) \text{ and } d_G(v, s) > d_G(v, t) , \\
                    \text{or} \\
                    d_G(u, s) > d_G(u, t) \text{ and } d_G(v, s) < d_G(v, t) , \\
                \end{cases} \\
            0 &\text{otherwise.}
        \end{cases}
\end{equation*}
This can be considered an indicator function that is 1 if and only if the vertices $u$ and $v$ contribute to $n_G(s, t) \, n_G(t, s)$. \citeauthor*{2017_OnTheDifferenceBetweenTheSzegedAndTheWienerIndex}~\cite{2017_OnTheDifferenceBetweenTheSzegedAndTheWienerIndex} used $\mu_{u, v}$ to rewrite the Szeged index in the following way:
\begin{equation*}
    \Sz(G) = \sum_{\{u, v\} \subseteq V(G)} \sum_{\{s, t\} \in E(G)} \mu_{u, v}(\{s, t\}) .
\end{equation*}
With this reformulation, the Szeged index is also a sum over all unordered pairs of vertices. Additionally, \citeauthor*{2017_OnTheDifferenceBetweenTheSzegedAndTheWienerIndex} called all edges $e$ satisfying $\mu_{u, v}(e) = 1$, `good' for $\{u, v\}$, and referenced this again to \citeauthor*{2000_SomeGraphsWithExtremalSzegedIndex}~\cite{2000_SomeGraphsWithExtremalSzegedIndex}. However, \citeauthor*{2000_SomeGraphsWithExtremalSzegedIndex} used the term `good edge' for a completely different concept. Because of this, and the fact that the term `good' is not descriptive, we decided to use a different notation. We call edges $e$ satisfying $\mu_{u, v}(e) = 1$, \emph{$(u, v)$-distance-disparate}, and denote with $\dis_G(u, v)$ the number of $(u, v)$-distance-disparate edges in $G$. Hence, we can write for the Szeged index,
\begin{equation*}
    \Sz(G) = \sum_{\{u, v\} \subseteq V(G)} \dis_G(u, v) = \frac{1}{2} \, \sum_{u, v \in V(G)} \dis_G(u, v).
\end{equation*}

Since the revised Szeged index may not even be an integer, there cannot be a single indicator function as there is for the Szeged index. So it seems difficult to formulate the revised Szeged index as sum over vertices. Still a rather similar approach works. The first step is to consider an equivalent of $\mu_{u, v}$ for single vertices and edges having end points with the same distance to the vertex. Namely, we define for a vertex $v$ and an edge $\{s, t\}$, 
\begin{equation*}
    \nu_v(\{s, t\}) = 
        \begin{cases}
            1 &\text{if } d_G(v, s) = d_G(v, t), \\
            0 &\text{otherwise,}
        \end{cases}
\end{equation*}
an indicator function that is 1 if and only if the end points of the edge have the same distance to $v$. Now, similar to before, we call edges $e$ satisfying $\nu_u(e) = 1$ and $\nu_v(e) = 1$ for vertices $u$ and $v$, \emph{$(u, v)$-distance-equal}, and denote with $\deq_G(u, v)$ the number of $(u, v)$-distance-equal edges in $G$. These are the ingredients necessary to write the revised Szeged index as sum over vertices. 

\begin{thm} \label{thm:RevisedSzAsVertexSum}
The revised Szeged index of a graph $G$ can be written as sum over vertices in the following form: 
\begin{equation*}
    \Sz^*(G) = \frac{1}{2} \, \sum_{u, v \in V(G)} 
        \br*{\dis_G(u, v) + \deq_G(u, u) - \frac{1}{2} \, \deq_G(u, v)} .
\end{equation*}
\end{thm}

\begin{proof}
Let $n$ be the number of vertices in $G$. Use that $n = n_G(s, t) + n_G(t, s) + o_G(s, t)$ for all edges $\{s, t\}$ to rewrite the revised Szeged index: 
\begin{equation} \label{eq:RevisedSzegedAsVertexSum}
    \begin{split}
        \Sz^*(G) &= \sum_{\{s, t\} \in E(G)} 
                \br*{n_G(s, t) + \frac{1}{2} \, o_G(s, t)} 
                \br*{n_G(t, s) + \frac{1}{2} \, o_G(t, s)} \\
            &= \sum_{\{s, t\} \in E(G)} 
                \br*{n_G(s, t) n_G(t, s) 
                    + \frac{1}{2} \, o_G(s, t) \br*{n - o_G(s, t)}
                    +  \frac{1}{4} \, o_G(s, t)^2} \\
            &= \Sz(G) 
                + \frac{1}{2} \, n \sum_{\{s, t\} \in E(G)} o_G(s, t) 
                - \frac{1}{4} \, \sum_{\{s, t\} \in E(G)} o_G(s, t)^2 . 
    \end{split}
\end{equation}
Since a vertex $v$ is counted in $o_G(s, t)$ if and only if $d_G(v, s) = d_G(v, t)$, we can rewrite the second sum to
\begin{equation*}
    \sum_{\{s, t\} \in E(G)} o_G(s, t) 
        = \sum_{u \in V(G)} \sum_{e \in E(G)} \nu_u(e)
        = \sum_{u \in V(G)} \deq_G(u, u) .
\end{equation*}
For the third term notice that vertices $u$ and $v$ are involved in $o_G(s, t) \cdot o_G(s, t)$ if and only if $d_G(u, s) = d_G(u, t)$ and $d_G(v, s) = d_G(v, t)$, that is $\{s, t\}$ is counted in $\deq_G(u, v)$. Thus, we can reformulate this sum as well:
\begin{equation*}
    \sum_{\{s, t\} \in E(G)} o_G(s, t)^2 = \sum_{u, v \in V(G)} \deq_G(u, v) .
\end{equation*}
Insert the reformulations and the Szeged index written as vertex sum in Equation~(\ref{eq:RevisedSzegedAsVertexSum}) and write for $n$ the sum over all vertices to get the desired result:
\begin{equation*}
    \begin{split}
        \Sz^*(G) &= \frac{1}{2} \, \sum_{u, v \in V(G)} \dis_G(u, v)
            + \frac{1}{2} \, n \sum_{u \in V(G)} \deq_G(u, u)
            - \frac{1}{4} \, \sum_{u, v \in V(G)} \deq_G(u, v) \\
            &= \frac{1}{2} \, \sum_{u, v \in V(G)} 
            \br*{\dis_G(u, v) + \deq_G(u, u) - \frac{1}{2} \, \deq_G(u, v)} .
    \end{split}
\end{equation*}
\end{proof}

A noteworthy consequence of the above result is that the difference between the Szeged and the revised Szeged index can be nicely described. 

\begin{cor} \label{cor:SzDiff}
The difference between the Szeged and the revised Szeged index of a graph $G$ on $n$ vertices satisfies
\begin{equation*}
    \begin{split}
        \Sz^*(G) - \Sz(G) &= \frac{1}{2} \, \sum_{\{s, t\} \in E(G)} 
            \br*{n \cdot o_G(s, t) - \frac{1}{2} \, o_G(s, t)^2} \\
            &= \frac{1}{2} \, n \sum_{u \in V(G)} \deq_G(u, u)
                - \frac{1}{4} \, \sum_{u, v \in V(G)} \deq_G(u, v) .
    \end{split}
\end{equation*}
\end{cor}

Before we come to the comparison of the Wiener index and the (revised) Szeged index on cactus graphs, we need some general results about graphs. The first is about the connection of $\dis_G$ and $d_G$ on cycles. 

\begin{lem} \label{lem:CycleDis2D}
Let $u$ and $v$ be two distinct vertices of a cycle $C$ of length $n$. Then
\begin{equation*}
    \dis_C(u, v) = 
                \begin{cases}
                    2 \, d_C(u, v)      & \text{if $n$ is even,} \\
                    2 \, d_C(u, v) - 1  & \text{if $n$ is odd.} \\
                \end{cases} 
\end{equation*}
\end{lem}

\begin{proof}
To make things easier, we think of a suitable embedding of $C$ in the plane and say right for counterclockwise, and left for clockwise. For some vertex $w$ in $C$, let $P_r(w)$ be the path starting at $w$ and going $\floor{\sfrac{n}{2}}$ edges to the right, and $P_l(w)$ the path starting at $w$, going $\floor{\sfrac{n}{2}}$ edges to the left. We denote the terminal vertices of $P_r(w)$ and $P_l(w)$ with $w_r$ and $w_l$, respectively. 

Let $e$ be an edge in $C$. It is clear that if $e$ is in $P_r(u)$, then the left vertex of $e$ is closer to $u$, and vice versa, if $e$ is in $P_l(u)$, then the right vertex of $e$ is closer to $u$. For $v$ the situation is the same. Thus, $e$ is $(u, v)$-distance-disparate if and only if it is contained in the path $P_r(u) \cap P_l(v)$, or in the path $P_l(u) \cap P_r(v)$. 

Without lost of generality, we can assume $v$ is in $P_r(u)$, see Figure~\ref{fig:EvenCycle} for an exemplary illustration of the situation. In this case, $P_r(u) \cap P_l(v)$ is a shortest path from $u$ to $v$, and $P_l(u) \cap P_r(v)$ is a shortest path from $u_l$ to $v_r$. So we have
\begin{equation} \label{eq:CycleWiSz}
    \dis_C(u, v) = d_C(u, v) + d_C(u_l, v_r). 
\end{equation}
By inclusion--exclusion principle, the distance from $u_l$ to $v_r$ can be determined by
\begin{equation*}
    \begin{split}
        d_C(u_l, v_r) &= \abs{P_l(u) \cap P_r(v)} \\
            &= \abs{P_r(u) \cap P_l(v)} + \abs{P_l(u)} + \abs{P_r(v)} - \abs{E(C)} \\
            &= d_C(u, v) + 2 \, \floor{\sfrac{n}{2}} - n .
    \end{split}
\end{equation*}
Now considering even and odd $n$ respectively, and inserting $d_C(u_l, v_r)$ in (\ref{eq:CycleWiSz}) completes the proof. 
\end{proof}

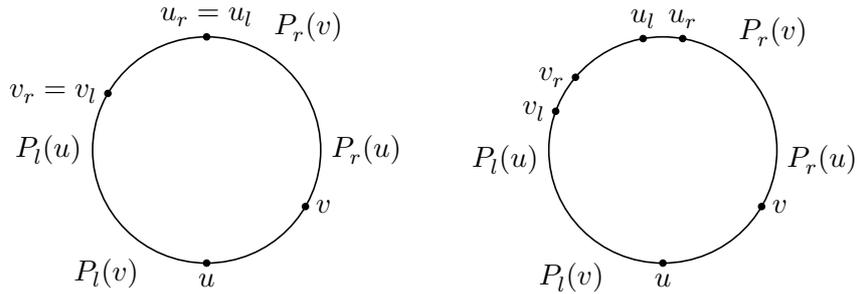
\begin{figure} [ht]
\centering
\begin{tikzpicture} [graph]
    \def\r{1.5cm}
    
    \begin{scope}[xshift=0cm]
        \begin{scope} [graph vertices]
        	\node [yshift = -\r] (u) {};
        	\node [yshift = \r] (ul) {};
        	
        	\draw (-30 : \r) node (v) {};
        	\draw (150 : \r) node (vr) {};
        \end{scope}
        
        \draw (0, 0) circle (\r);
        
    	\node at (u) [below] {$u$};
    	\node at (ul) [above] {$u_r = u_l$};
    	
    	\node at (v) [right] {$v$};
    	\node at (vr) [left] {$v_r = v_l$};
    	
    	\draw (0 : \r)  node [right] {$P_r(u)$};
    	\draw (180 : \r) node [left] {$P_l(u)$};
    	\draw (60 : \r)  node [above right] {$P_r(v)$};
    	\draw (240 : \r) node [below left] {$P_l(v)$};
    \end{scope}
    
    \begin{scope}[xshift=6cm]
        \begin{scope} [graph vertices]
        	\node [yshift = -\r] (u) {};
        	\draw (80 : \r) node (ur) {};
        	\draw (100 : \r) node (ul) {};
        	
        	\draw (-30 : \r) node (v) {};
        	\draw (140 : \r) node (vr) {};
        	\draw (160 : \r) node (vl) {};
        \end{scope}
        
        \draw (0, 0) circle (\r);
        
    	\node at (u) [below] {$u$};
    	\node at (ur) [above] {$u_r$};
    	\node at (ul) [above] {$u_l$};
    	
    	\node at (v) [right] {$v$};
    	\node at (vr) [left] {$v_r$};
    	\node at (vl) [left] {$v_l$};
    	
    	\draw (-5 : \r)  node [right] {$P_r(u)$};
    	\draw (185 : \r) node [left] {$P_l(u)$};
    	\draw (55 : \r)  node [above right] {$P_r(v)$};
    	\draw (245 : \r) node [below left] {$P_l(v)$};
    \end{scope}
\end{tikzpicture}
\caption{Cycle $C$ of even length left and of odd length right, with vertices $u$, $v$, and the paths going right and left including their terminal vertices.}
\label{fig:EvenCycle}
\end{figure}

The next result is about splitting distances in a block-cut-vertex decomposition of the given graph. Recall, a block is a maximal 2-connected subgraph, and in a block-cut-vertex decomposition blocks only overlap at cut vertices. More information on blocks and the block-cut-vertex decomposition can be found in \cite{1984_Book_GraphTheory}. 

\begin{prop} \label{prop:DistancePartition}
Let $u$ and $v$ be vertices of a graph $G$ with set of blocks $\mathcal{B}$ obtained by the block-cut-vertex decomposition for $G$. For a block $B$ in $\mathcal{B}$, denote by $u_B$ and $v_B$ the vertices in $B$ closest to $u$ and $v$, respectively. Then
\begin{equation}
    d_G(u, v) = \sum_{B \in \mathcal{B}} d_G(u_B\,, v_B) .
\end{equation}
\end{prop}

\begin{proof}
Let $P$ be a shortest path from $u$ to $v$ and $\mathcal{B}' \subseteq \mathcal{B}$ the set of blocks visited by $P$. Every block $B$ in $\mathcal{B}'$ is entered by $u_B$ and left by $v_B$\,, so $P$ can be decomposed into subpaths $P_B$, where for a block $B$ the subpath $P_B$ starts at $u_B$ and ends at $v_B$\,. Since every subpath of a shortest path is a shortest path itself, it follows that
\begin{equation}
    d_G(u, v) = \abs{P} = \sum_{B \in \mathcal{B}'} \abs{P_B} = \sum_{B \in \mathcal{B}'} d_G(u_B\,, v_B) .
\end{equation}

Now consider a block $B$ not visited by $P$. Since we have a block-cut-vertex decomposition, there is a unique vertex $w$ in $B$ minimizing the distance from the block $B$ to the path $P$. This vertex is also a cut vertex and thus it minimises the distance from $B$ to any vertex of the path $P$. Hence, it follows that $u_B = v_B = w$, and $d_G(u_B\,, v_B) = 0$. This finishes the proof. 
\end{proof}

Note, in the proof of Proposition~\ref{prop:DistancePartition} we do not use that blocks are two-connected. That means, instead of blocks, we could split the graph into arbitrary subgraphs that only overlap at cut vertices. 

In the remaining two sections, we apply the above tools to the so called cactus graphs. These are connected graphs where every two distinct cycles have at most one common vertex. Alternatively, the graph consists of a single vertex, or every block is either an edge or a cycle.

\section{Comparing Wiener and Szeged index on cactus graphs}
\label{sec:WIAndSzOnCactusGraphs}

Since every edge on a shortest path from $u$ to $v$ is clearly $(u, v)$-distance-disparate, formulating the Szeged index as sum over vertices gives the first part of the following inequality: 
\begin{equation*}
    W(G) \leq \Sz(G) \leq \Sz^*(G) .
\end{equation*}

Already \citeauthor*{2000_SomeGraphsWithExtremalSzegedIndex} used the indicator function $\mu_{u, v}$ to show additionally that equality holds in the first part if and only if every block of $G$ is complete, see \cite[Theorem~2.1]{2000_SomeGraphsWithExtremalSzegedIndex}. The inequality of the second part is clear by definition, whereas equality holds if and only if $G$ is bipartite. This was shown by \citeauthor{2010_UseOfTheSzegedIndexAndTheRevisedSzegedIndexForMeasuringNetworkBipartivity}, see \cite[Theorem 1]{2010_UseOfTheSzegedIndexAndTheRevisedSzegedIndexForMeasuringNetworkBipartivity}. Besides, it follows from Corollary~\ref{cor:SzDiff}. 

Here, we want to show a different inequality, true for the special class of cactus graphs. 

\begin{thm} \label{thm:Sz2W}
Let $G$ be a cactus graph, then 
\begin{equation*}
    \Sz(G) \leq 2 \, W(G), 
\end{equation*}
with equality if and only if every block of $G$ is a cycle of even length. 
\end{thm}

A special case of this result was already given in \cite{2016_ProofsOfThreeConjecturesOnTheQuotientsOfTheSzegedIndexAndTheWienerIndex}. There, \citeauthor{2016_ProofsOfThreeConjecturesOnTheQuotientsOfTheSzegedIndexAndTheWienerIndex} showed that Theorem~\ref{thm:Sz2W} holds for unicyclic graphs. 

\begin{proof}
Let $u$, $v$ be vertices in $G$ and $e$ be an edge in a block $B$. With $u_B$ as in Proposition~\ref{prop:DistancePartition} every shortest path from $e$ to $u$ uses $u_B$. The same is true for $v$ and $v_B$ as in Proposition~\ref{prop:DistancePartition}, respectively. Thus $e$ is $(u, v)$-distance-disparate if and only if it is $(u_B\,, v_B)$-distance-disparate. Hence, with $\mathcal{B}$ as set of blocks, we can write
\begin{equation}
    \dis_G(u, v) = \sum_{B \in \mathcal{B}} \dis_G(u_B\,, v_B) . 
\end{equation}

Suppose that every block is a cycle of even length. Then by Lemma~\ref{lem:CycleDis2D} and Proposition~\ref{prop:DistancePartition},
\begin{equation}
    \begin{split}
        \Sz(G) &= \frac{1}{2} \sum_{u, v \in V(G)} \dis_G(u, v) 
        = \frac{1}{2} \sum_{u, v \in V(G)} 
            \sum_{B \in \mathcal{B}} \dis_G(u_B\,, v_B) \\
        &= \frac{1}{2} \sum_{u, v \in V(G)} \sum_{B \in \mathcal{B}} 2 d_G(u_B\,, v_B)
        = \sum_{u, v \in V(G)} d_G(u, v) \\ 
        &= 2 \, W(G) .
    \end{split}
\end{equation}

Now if there is at least one odd cycle $C$, then again by Lemma~\ref{lem:CycleDis2D}, there is a strict inequality instead of the third equality in the above formula. Finally, if there is a block consisting of only a single edge $\{s, t\}$, then $\dis_G(s, t) = 1 = d_G(s, t)$, and thus also $\Sz(G) < 2 \, W(G)$. 
\end{proof}

Note, blocks consisting of two vertices connected with two edges considered as cycles of length two can be allowed in Theorem~\ref{thm:Sz2W}. Clearly, this is not a characterisation of graphs $G$ satisfying $\Sz(G) \leq 2 \, W(G)$, since every complete graph $K_n$ satisfies $\Sz(K_n) = W(K_n)$. Unfortunately, it is also not a characterisation of graphs satisfying $\Sz(G) = 2 \, W(G)$. Below, we give an example of a graph satisfying the equation that is not a cactus graph. 

\begin{exam} \label{exam:Sz2W}
Let $G$ consist of three paths of length two joined at their end points. Attach on one side of the end points of the paths two edges by their end points and on the other side three edges. See Figure~\ref{fig:CounterExampleSzeged} for an exemplary drawing. It can be checked via a computer, or even easily by hand that
\begin{equation*}
    \Sz(G) = 192 = 2 \cdot 96 = 2 \, W(G) .
\end{equation*}

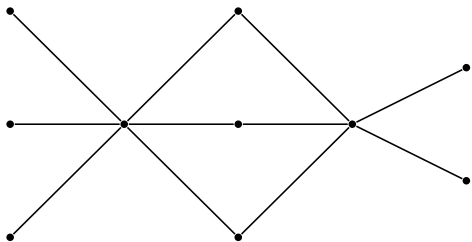
\begin{figure} [ht]
\centering
\begin{tikzpicture} [graph]
    \def\r{1.5cm}
    
    \begin{scope} [graph vertices]
    	\node [xshift = -2 * \r, yshift = -\r] (l1) {};
    	\node [xshift = -2 * \r] (l2) {};
    	\node [xshift = -2 * \r, yshift = \r] (l3) {};
    	
    	\node [xshift = -\r](cl) {};
    	
    	\node [yshift = -\r] (c1) {};
    	\node (c2) {};
    	\node [yshift = \r] (c3) {};
    	
    	\node [xshift = \r] (cr) {};
    	
    	\node [xshift = 2 * \r, yshift = -\r / 2] (r1) {};
    	\node [xshift = 2 * \r, yshift = \r / 2] (r2) {};
    \end{scope}
    
    \draw (l1) edge (cl);
    \draw (l2) edge (cl);
    \draw (l3) edge (cl);
    
    \draw (c1) edge (cl);
    \draw (c2) edge (cl);
    \draw (c3) edge (cl);
    
    \draw (c1) edge (cr);
    \draw (c2) edge (cr);
    \draw (c3) edge (cr);
    
    \draw (r1) edge (cr);
    \draw (r2) edge (cr);
\end{tikzpicture}
\caption{A bipartite non-cactus graph $G$ satisfying $\Sz(G) = 2 \, W(G)$.}
\label{fig:CounterExampleSzeged}
\end{figure}
\end{exam}

By generalizing the graph in Example~\ref{exam:Sz2W} to have $k$ paths instead of only 3, more example graphs satisfying the equality can be found. Not for every $k$ a suitable number of edges can be attached, but it seems there is no cap for $k$. The biggest example graph $G$ we found has 783 paths of length 2, 28 edges on one and 656\,009 edges on the other side attached. It satisfies 
\begin{equation*}
    \Sz(G) = 862\,902\,435\,600 = 2 \cdot 431\,451\,217\,800 = 2 \, W(G).
\end{equation*}
This suggests that also if the cyclomatic number, which is just $\abs{E(G)} - \abs{V(G)} + 1$ for connected graphs, is large, $\Sz(G) = 2 \, W(G)$ can still hold for non-cactus graphs. 

\section{Comparing Wiener and revised Szeged index on cactus graphs}
\label{sec:WIAndRevisedSzOnCactusGraphs}

From the last section, we can conclude that $\Sz^*(G) \leq 2 \, W(G)$ holds for bipartite cactus graphs $G$. But in case of non-bipartite cactus graphs the situation becomes more complicated. There are even cactus graphs $G$ satisfying $\Sz^*(G) = 2 \, W(G)$, where not every block is a cycle of even length as the following example shows. 

\begin{exam} \label{exam:RevisedSz2W}
Take a cycle of length 13, a cycle of length 11, six edges and join them at a single vertex to obtain a cactus graph $G$, as depicted in Figure~\ref{fig:CounterExampleRevisedSzeged}. It can be checked that
\begin{equation*}
    \Sz^*(G) = 3636 = 2 \cdot 1818 = 2  \, W(G) .
\end{equation*}

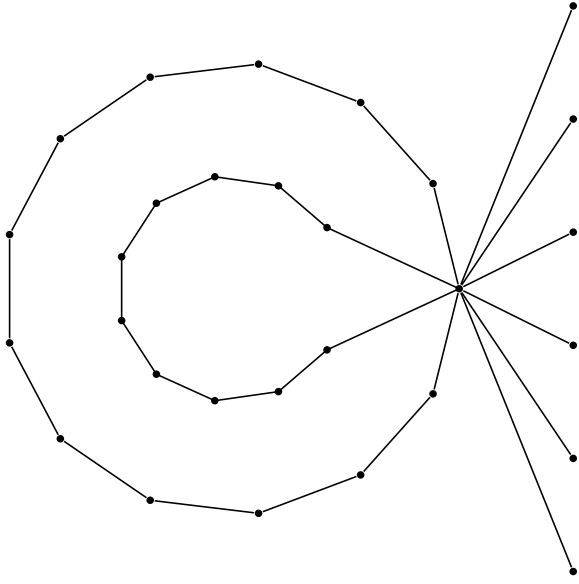
\begin{figure} [ht]
\centering
\begin{tikzpicture} [graph]
    \def\r{1.5cm}
    \begin{scope} [graph vertices]
		\foreach \i in {1, ..., 10}
		{
			\draw (\i * 360 / 11 : \r) node (i\i) {};
		}
		
		\foreach \i in {0, ..., 12}
		{
			\draw (\i * 360 / 13 : 2 * \r) node (o\i) {};
		}
		
		\foreach \i in {1, ..., 6}
		{
			\node at (3 * \r, 3.5 * \r - \i * \r) (e\i) {};
		}
    \end{scope}
	
	\foreach \i in {1, ..., 9}
	{
		\draw (\i * 360 / 11 : \r) edge ({(\i + 1) * 360 / 11} : \r);
	}
    
    \foreach [evaluate = {\j = int(mod(\i + 1, 13))}] \i in {0, ..., 12}
	{
		\draw (o\i) edge (o\j);
	}
    
    \foreach \i in {1, ..., 6}
	{
		\draw (o0) edge (e\i);
	}
	
	\draw ($(i1)+(0,0)$) edge ($(o0)+(0,0)$);
	\draw ($(i10)+(0,0)$) edge ($(o0)+(0,0)$);
    
\end{tikzpicture}
\caption{A cactus graph $G$ satisfying $\Sz^*(G) = 2 \, W(G)$.}
\label{fig:CounterExampleRevisedSzeged}
\end{figure}
\end{exam}

With this in mind, it seems difficult to make any concrete statements about the connection of the revised Szeged and the Wiener index in the case of cactus graphs. Hence, we focused on a subclass of cactus graphs and found the following relation, which is in contrast to Theorem~\ref{thm:Sz2W}. 

\begin{thm} \label{thm:RevisedSz2W}
Suppose every block of a graph $G$ is a cycle. Then 
\begin{equation*}
    2 \, W(G) \leq \Sz^*(G) , 
\end{equation*}
with equality if and only if every cycle in $G$ has even length. 
\end{thm}

Note, clearly a graph where every block is a cycle is a cactus graph. 

\begin{proof}
Let $u$, $v$ be vertices in $G$ and $e$ be an edge in a block $B$. Again, we use the notation of Proposition~\ref{prop:DistancePartition} with $u_B$ and $v_B$ for the vertices in $B$ closest to $u$ and $v$, respectively. Since $u_B$ is on every shortest path from $u$ to $e$, and the same is true for $v_B$ and $v$, it is evident that $\deq_B(u, v) = \deq_B(u_B, v_B)$. Furthermore, the set of blocks $\mathcal{B}$ of $G$ induces a partition of the edge set. Hence, 
\begin{equation*}
    \deq_G(u, v) = \sum_{B \in \mathcal{B}} \deq_B(u, v) = \sum_{B \in \mathcal{B}} \deq_B(u_B, v_B) .
\end{equation*}
Thus, with Theorem~\ref{thm:RevisedSzAsVertexSum} we can formulate the revised Szeged index of $G$ as
\begin{equation} \label{eq:RevisedSzBlockSum}
    \begin{split}
        \Sz^*(G) &= \frac{1}{2} \, \sum_{u, v \in V(G)} 
            \br*{\dis_G(u, v) + \deq_G(u, u) - \frac{1}{2} \, \deq_G(u, v)} \\
            &= \frac{1}{2} \sum_{u, v \in V(G)} 
                \sum_{B \in \mathcal{B}} 
                    \br*{\dis_G(u_B, v_B) + \deq_B(u_B, u_B) - \frac{1}{2} \, \deq_B(u_B, v_B)} .
    \end{split}
\end{equation}

Next we distinguish two cases, whereby the second case has two sub-cases, to show that for any vertices $u_B$ and $v_B$ in a block $B$, 
\begin{equation} \label{eq:RevisedSzSummand}
    2 \, d_G(u_B, v_B) \leq \dis_G(u_B, v_B) + \deq_B(u_B, u_B) - \frac{1}{2} \, \deq_B(u_B, v_B) .
\end{equation}
Case 1: Suppose that $B$ is a cycle of even length. Then,
\begin{align*}
    \deq_B(u_B, u_B) &= 0 = \deq_B(u_B, v_B), 
\shortintertext{and by Lemma~\ref{lem:CycleDis2D},}
    \dis_G(u_B, v_B) &= 2 \, d_G(u_B, v_B).
\end{align*}
Case 2: Suppose that $B$ is a cycle of odd length. 

Case 2.1: If $u_B \neq v_B$, then 
\begin{align*}
    \deq_B(u_B, u_B) &= 1, \quad 
        \deq_B(u_B, v_B) = 0, 
\shortintertext{\hspace{\parindent}and again by Lemma~\ref{lem:CycleDis2D}}
    \dis_G(u_B, v_B) &= 2 \, d_G(u_B, v_B) - 1.
\end{align*}
            
Case 2.2: If $u_B = v_B$, then
\begin{equation*}
    \begin{split}
    \deq_B(u_B, u_B) &= 1 = \deq_B(u_B, v_B), \\
    \dis_G(u_B, v_B) &= 0 = 2 \, d_G(u_B, v_B).  
    \end{split}
\end{equation*}

So in Case~1 and Case~2.1, we have equality in (\ref{eq:RevisedSzSummand}), and in Case~2.2, (\ref{eq:RevisedSzSummand}) is fulfilled with a strict inequality. Therefore by Proposition~\ref{prop:DistancePartition} and (\ref{eq:RevisedSzBlockSum}),
\begin{equation*}
    2 \, W(G) = \frac{1}{2} \, \sum_{u, v \in V(G)} \sum_{B \in \mathcal{B}} 2 \, d_G(u_B, v_B) 
        \leq \Sz^*(G) , 
\end{equation*}
with equality if and only if every cycle in $G$ has even length. 
\end{proof}

\section{Acknowledgements}

Stefan Hammer acknowledges the support of the Austrian Science Fund (FWF): W1230. 

\printbibliography

\end{document}